\documentclass[10pt]{amsart}
\usepackage{amsmath,amsthm,amsfonts}
\usepackage{graphicx}
\usepackage{hyperref}

\usepackage{algorithm}
\usepackage{algorithmic}
\newtheorem{alg}{Algorithm}
\renewenvironment{algorithm}[1]
{\vspace{10pt}
\begin{alg}\rm
{#1}
\vspace{5pt}
\hrule
\vspace{5pt}}
{\vspace{5pt}
\hrule
\end{alg}
\vspace{10pt}}

\algsetup{indent=2em}

\newtheorem{theorem}{Theorem}[section]
\newtheorem{lemma}[theorem]{Lemma}
\theoremstyle{definition}

\newtheorem{problem}[theorem]{Problem}

\newtheorem{proposition}[theorem]{Proposition}

\newtheorem*{thm}{Theorem}
\newtheorem*{cor}{Corollary}

\newcommand{\gp}[1]{{\left\langle #1 \right\rangle}}
\newcommand{\rb}[1]{{\left( #1 \right)}}

\def\CN{{\mathcal N}}
\def\CP{{\mathcal P}}
\def\CE{{\mathcal E}}
\def\CR{{\mathcal R}}
\def\CS{{\mathcal S}}
\def\MZ{{\mathbb{Z}}}
\def\MN{{\mathbb{N}}}

\title[Word Problem for the Baumslag group $G_{(1,2)}$]{The Word Problem in the Baumslag group with a non-elementary Dehn function is polynomial time decidable}

\author[]{Alexei Miasnikov}
\address{Department of Mathematics, Stevens Institute of Technology,
Hoboken, NJ, USA} \email{amiasnikov@gmail.com}
\thanks{The work of the first and the second author was partially supported by the NSF grant DMS-0914773.}

\author[]{Alexander Ushakov}
\address{Department of Mathematics, Stevens Institute of Technology,
Hoboken, NJ, USA} \email{sasha.ushakov@gmail.com}

\author[]{Dong Wook Won}
\address{Department of Mathematics, CUNY/LAGCC, Long Island City, NY, USA} \email{dwwon@hotmail.com}

\begin{document}

\maketitle

\begin{abstract}
We prove that the Word problem in the  Baumslag  group
$G_{(1,2)} = \gp{a,b ~;~ a^{a^b}=a^2}$ which has a non-elementary Dehn function is decidable in polynomial time.


\end{abstract}



\section{Introduction}

One-relator groups form a very interesting and very mysterious  class of groups.
In 1910s Dehn proved that the Word problem for the standard presentation of the fundamental group of a closed oriented surface of genus at least two is solvable by what is now called Dehn's algorithm (see \cite{MC} for details).  In 1932 Magnus developed a general powerful approach  to one-relator groups \cite{Magnus:1932},
nowadays known as Magnus  break-down procedure
(see \cite{mks,LS}). In particular, he solved the Word problem (WP) 
in an arbitrary one-relator group. The decision algorithm is quite 
complicated and its time complexity is unknown.  In fact, we show 
here that the time function of the Magnus decision algorithm on the Baumslag group
    $$G_{(1,2)} = \gp{a,b \mid b^{-1} a^{-1} b a b^{-1} a b = a^2 }$$
is not bounded by any finite tower of exponents. Furthermore, it is unknown whether there exists any  feasible general (uniform)  algorithm that solves WP in all one-relator groups, and at present it seems implausible that such algorithm exists. However, it is quite possible that the Word problem in every fixed one-relator group is tractable.  In the Magnus collection of open problems in groups theory \cite{BMS} the following question is posted.

 \begin{problem} [\cite{BMS}, (OR3)]
Is it true that WP in every given one-relator group $G$ is decidable in polynomial time?
\end{problem}

The current state of affairs on WP in one-relator groups can be described as follows.
On one hand, there are  several large classes of one-relator groups where WP is well understood and is decidable in polynomial time (hyperbolic, automatic, linear, etc). On the other hand,   there are several sporadic examples of one-relator groups where WP  requires a special treatment, though at the end is polynomial time decidable. Finally, there is  a  few one-relator groups  where WP seems especially hard and the time complexity is unknown.
These are the most interesting ones in this context.

One of the principal unsolved mysteries on  one-relator groups  is which of them have  a hard WP and why. More precisely, the problem is to determine the ``general classes'' of one-relator groups and  divide the rest (the sporadic, exceptional ones) into  some well-defined families.

 There are several
 conjectures that describe large general  classes of one-relator groups  which we would like to mention here.

 \subsection{Hyperbolic groups}
 Notice, that if $G$  is hyperbolic, in particular, if it satisfies the small cancelation condition $C'(\frac{1}{6})$, then WP in $G$ is decidable in linear time by  Dehn's   algorithm \cite{Gr1}.
 Since the asymptotic density of the set of words $w \in F(X)$ for which  the symmetrized one-relator  presentation   $\langle X \mid w\rangle$ is $C'(\frac{1}{6})$ small cancelation is equal to 1, one may say that for generic one-relator groups the answer to the question above is affirmative. One can check in polynomial (at most quadratic) time if a one-relator  presentation, when symmetrized, is $C'(\frac{1}{6})$ or not. Hence, it is possible to run in parallel the Magnus break-down process and the Dehn's algorithm  for symmetrized $C'(\frac{1}{6})$ presentations and obtain a correct uniform total algorithm that solves WP in one-relator groups, and has Ptime complexity on the set of one-relator groups of asymptotic density 1. Unfortunately, such an algorithm will not be feasible on the most  interesting examples of one-relator groups.
Some interesting examples of hyperbolic one-relator groups can be found in \cite{Ivanov-Schupp:1998}.

Of course, not all one-relator groups are hyperbolic. The famous Baumslag-Solitar one-relator groups
    $$B_{(m,n)} = \gp{a,b \mid b^{-1}a^mb = a^n }, m,n \geq 1,$$
introduced in \cite{Baumslag_Solitar:1962} are not hyperbolic,
since the groups $B_{(1,n)}$ are infinite metabelian, and the other ones
contain  $F_2 \times \mathbb{Z}$ as a subgroup.

The following outstanding conjecture (see \cite{BMS}) describes, if true, one-relator hyperbolic groups.

\begin{problem}
Is every one-relator group without Baumslag-Solitar subgroups hyperbolic?
\end{problem}

 Independently of the above, it is very interesting to know which one-relator groups contain groups $B_{(m,n)}$.

 \begin{problem}
 Is there an algorithm to recognize if a given one-relator group contains a subgroup $B_{(m,n)}$ for some $m,n \geq 1$?
 \end{problem}

Notice, that in 1968 B.~B.~ Newman in \cite{BB_Newman:1968} showed that all one-relator groups with
torsion are hyperbolic and, hence, the Word problem for them is decidable in linear time.

\subsection{Automatic groups}
Automatic groups form another class  where WP is easy. It is known   that every hyperbolic group is automatic and WP is decidable in at most quadratic time in a given automatic group.
Furthermore, the Dehn function in automatic groups is quadratic. We refer to \cite{Epstein}
for more details on automatic groups.
Observe, that the group $B_{(m,n)}$ is not automatic provided $m \neq n$,
since its Dehn function is exponential.

The main challenge in this area is to describe one-relator automatic groups. Answering the following questions would help to understand which one-relator groups are automatic.

\begin{problem}
Is it true that one-relator groups with a quadratic Dehn function are automatic?
\end{problem}

\begin{problem}
Is it true that one-relator groups with no subgroups isomorphic to $B_{(m,n)}, m\neq n$ are automatic?
\end{problem}

\begin{problem}[\cite{BMS} (OR8)]
Is the one-relator group $\langle X \mid [u,v]\rangle$ automatic for any words $u,v \in F(X)$?
\end{problem}

\subsection{Linear and residually finite groups}
Lipton and Zalstein in \cite{Lipton-Zalstein:1977}
proved that WP in linear groups is polynomial time decidable,
so one-relator linear groups provide a general
subclass of one-relator groups where WP is easy.  Until recently, not much was known about linearity of one-relator groups. We refer to \cite{Baumslag:1974}
for an initial discussion that formed the area for years to come.  The real breakthrough came in 2009 when Wise announced in \cite{Wise:2009}
that if a hyperbolic group $G$ has a quasi-convex hierarchy then it is virtually
a subgroup of a right angled Artin group and, hence, is linear.
This result covers a lot of one-relator groups, in particular all one-relator groups with torsion.
There are two interesting cases that we would like to  mention here.
In \cite{Baumslag:1962}
Baumslag introduced {\em cyclically pinched} one-relator groups as those ones that can be presented as a free product of free groups with cyclic amalgamation
    $$\langle X \cup Y \mid u = v\rangle = F(X) \ast_{u = v} F(Y)$$
where $u \in F(X)$ and $v \in F(Y)$ are non-trivial non-primitive elements in the corresponding factors. Similarly, one can define {\em conjugacy pinched} one-relator groups as HNN extensions of free groups with cyclic associated subgroups:
    $$\langle F(X), t \mid t^{-1}ut = v\rangle$$
Wehrfritz proved in \cite{Wehrfritz:1973b}
that if non of $u$ and $v$ are proper powers then the group $F(X) \ast_{u = v} F(Y)$ is linear.
However, it was shown in \cite{Bestvina_Feighn:1992,Kharlampovich_Myasnikov:1998(3)}
that if either $u$ or $v$ is not a proper power then the group $F(X) \ast_{u = v} F(Y)$ is hyperbolic, so WP in these groups is linear time decidable.  Similar results hold for conjugacy pinched one-relator groups as well.  Observe, that cyclically and conjugacy pinched one-relator hyperbolic groups have quasi-convex hierarchy, so their linearity follows from Wise's result.
On the other hand, WP in hyperbolic groups is easy anyway,
so linearity in this case  does not give  much in terms of the efficiency of WP.

The general problem which one-relator non-hyperbolic groups are linear is wide open.   Recall, that  every finitely generated linear group is residually finite. Hence,  to see that a given one-relator groups is not linear  it suffices to show that it is not  residually finite.

Notice that there is a special decision algorithm  for WP  in  residually finite finitely presented groups.
The algorithm when given such a group $\langle X \mid R\rangle$  and a word $w \in F(X)$ runs two procedures in parallel: the first one enumerates  all the consequences of the relators $R$ until the word $w$ occurs, in which case $w = 1$ in $G$; while the second one  checks if $w$  is non-trivial in some finite quotient of $G$. Since $R$ is finite and $G$ is residually finite,
one of the two procedures eventually stops and gives the solution of WP  for $w$.
However, this algorithm is extremely inefficient.
This is why we do not discuss residually finite one-relator groups as a separate class here,
but only briefly mention the results that are related to linearity.

Meskin in \cite{Meskin:1972} studied residual finiteness of the following special class of one-relator groups:
    $$B(u,v,m,n) = \gp{X \mid u^{-1}v^mu = v^n }, m,n \geq 1,$$
where $u$ and $v$ are arbitrary non-commuting elements in $F(X)$.
He showed that if $m \neq 1, n \neq 1, m \neq n$ then the group $B(u,v,m,n)$ is not residually finite.
It follows that the group
$B_{(m,n)}$ is residually finite if and only if $m = 1$,  or $n = 1$, or $ m = n$.

Later  Vol'vachev in \cite{Vol'vachev:1985} found linear representation for all residually finite groups $B(u,v,m,n)$.
Sapir and Drutu constructed in \cite{Drutu_Sapir:2005} the first example of residually finite non-linear one-relator groups. They showed that the group
    $$DS = \langle a,t \mid t^{-2}a{t^2} = a^2\rangle $$
is residually finite and non-linear.

 The general classes of one-relator groups  described above are the only known ones where WP is polynomial time decidable. Now we describe the known sporadic one-relator groups where WP is presumably hard or requires a special approach.

\subsection{Baumslag-Solitar groups}
Gersten showed that the groups $B_{(m,n)}$, where $m\neq n$, have exponential
Dehn functions \cite{Ge91} (see also \cite{Epstein} and \cite{Groves_Hermiller:2001}),
so they are not hyperbolic or automatic.  As we mentioned above the metabelian groups $B_{(1,n)}$ are linear, so WP in them is polynomial time decidable.
The non-metabelian groups  $B_{(m,n)}$ are not linear, so WP in them requires a special approach. Nevertheless, WP  in these groups is polynomial time decidable (see Section \ref{se:HNN_extension}).
It would be interesting to study WP in the groups $B(u,v,m,n)$ which are similar to the Baumslag-Solitar groups.

\begin{problem}
What is complexity of WP in the groups $B(u,v,m,n)$ ?
\end{problem}

\subsection{Baumslag group $G_{(1,2)}$}

The group $G_{(1,2)}  = \gp{a,b \mid b^{-1} a^{-1} b a b^{-1} a b = a^2 }$ is truly remarkable.
Baumslag introduced this group in \cite{Baumslag:1969} and showed that all its finite quotients are cyclic.
In particular, the group $G_{(1,2)}$ is not residually finite and, hence, is not linear.
In \cite{G1} Gersten showed that the Dehn function for $G_{(1,2)}$ is not elementary, since it has the lower bound
$tower_2(\log_2 (n))$ and later Platonov in \cite{Platonov} proved that
$tower_2(\log_2 (n))$ is exactly the Dehn function for $G_{(1,2)}$.  This shows that $G_{(1,2)}$ is not hyperbolic, or automatic, or asynchronously automatic.  It was conjectured by Gersten that  $G_{(1,2)}$  has the highest Dehn function among all one-relator groups. As we have mentioned above the time function for the  Magnus break-down algorithm on $G_{(1,2)}$  is not elementary. Taking this into account it was believed until recently that WP in $G_{(1,2)}$ is among the hardest to solve among all one-relator groups.
In this paper we show that the Word
problem for $G_{(1,2)}$ can be solved in polynomial time.  To this end we develop a new technique to compress general exponential polynomials in the base 2 by algebraic circuits (straight-line programs) of a very special type, termed {\em power circuits} \cite{Miasnikov_Ushakov_Won_1}. We showed that one can do many standard algebraic manipulations (operations $x+y,x-y, x\cdot 2^y, x \leq y$) over the values of exponential polynomials, whose  standard binary length is not bounded by a fixed towers of exponents, in polynomial time if  it is kept in the compressed form.  This enables us to perform  some variations of  the standard algorithms in HNN extensions (or similar groups) keeping the actual rewriting in the compressed form. The resulting algorithms are of   polynomial  time, even though the standard versions are non-elementary.

\subsection{Baumslag groups $G_{(m,n)}$}

The  approach outlined above is quite general and we believe it can be useful elsewhere.
In particular, it works for groups of the type $G_{(m,n)}$,  where $m$ divides $n$.
Here the groups  $G_{(m,n)}$ are defined by the following presentations:
    $$G_{(m,n)}  = \gp{a,b \mid b^{-1} a^{-1} b a^m b^{-1} a b = a^n }.$$
Unfortunately, we do not have any compression techniques for the case when $m$ does not divide $n$
and $n$ does not divide $m$.
So the following problem seems currently as the main challenge regarding WP in one-relator groups.

\begin{problem}
What is the time-complexity of the Word problem for $G_{(2,3)}$?
\end{problem}

\subsection{Generalized Baumslag groups}

In \cite{Baumslag_Miller_Troeger:2007} Baumslag,  Miller and Troeger
studied another series of one-relator groups $G(r,w)$ which are similar to the group $G_{(1,2)}$. Namely,
if  $r, w$ are two non-commuting words in $F(X)$ then put
$$G(r,w) = \langle X \mid r^{r^w} = r^2\rangle.$$
The group $G(r,w)$ is not residually finite (neither linear nor  hyperbolic), it  has precisely the same finite quotients as the group $\langle X \mid r \rangle$.  These groups surely among the ones with non-easy WP.

\begin{problem}
Let $r,w$ be two non-commuting elements in $F(X)$.
\begin{itemize}
\item [1)] What is the Dehn function of  $G(r,w)$?
\item [2)] What is time complexity of WP in $G(r,w)$?

\end{itemize}
\end{problem}

Going a bit further one can consider WP in the following groups
$$G(r,w,m,n) = \langle X \mid (r^m)^{r^w} = r^n\rangle.$$

%
%
%
%
%
%
%

The paper is organized as follows. 
In section \ref{se:HNN_extension} we discuss algorithmic properties of
elements of $G_{(1,2)}$ as an HNN extension of the Baumslag-Solitar group,
set up the notation, and outline the difficulty of solvig the Word
problem using the standard methods for HNN extensions. In Section
\ref{se:power_circuits} we define the main tool in our method,
namely the power circuits, and present techniques for working with
them. In Section \ref{se:power_sequences} we define a
representation for words over some alphabet which we call a power
sequence. In Section \ref{se:algorithm} we present the algorithm
for solving the Word problem in $G_{(1,2)}$ and prove that its time-complexity is
$O(n^7)$.

\section{The group $G_{(1,2)}$}
\label{se:HNN_extension}

In this section we represent the group $G_{(1,2)}$ as an HNN extension of the Baumslag-Solitar group $B_{(1,2)}$ and describe two rewriting systems $\CR$ and $\CR'$ to solve WP in $G_{(1,2)}$. The system $\CR$  represents the classical Magnus breakdown algorithm for $G_{(1,2)}$. To study complexity of rewriting with $\CR$ we construct  an infinite sequence of words $\{w_k\}$ such  that
\begin{itemize}
\item $|w_k| \leq 2^{k+2}$;
\item it takes at least $tower_2(k-1)$ steps for $\CR$ to rewrite $w_k$;
\item  $\CR$ rewrites $w_k$ into a unique word of length $tower_2(k)$.
\end{itemize}
This shows, in particular, that the time function of the Magnus breakdown algorithm on $G_{(1,2)}$ is not bounded by any finite tower of exponents. Our strategy to solve WP in $G_{(1,2)}$ can be roughly described as follows. We combine many elementary steps in rewriting by $\CR$ into a single giant step and make  it an elementary rewrite of a new system $\CR'$.
 It is not hard to see that now  it takes only polynomially many steps for $\CR'$ to solve WP in $G_{(1,2)}$. In the rest of the paper we show that every elementary rewrite in $\CR'$ (the giant step)
can be done in polynomial time in the length of the input, thus proving that WP in $G_{(1,2)}$ is decidable in polynomial time.

\subsection{HNN extensions}

The purpose of this section is to introduce notation and the technique that we use throughout the paper.

 Let $H$ be  a group with two isomorphic subgroups $A$ and $B$,
 and $\varphi:A \rightarrow B$ an isomorphism. Then the group
$$G = \gp{ H , t \mid a^t = \varphi(a) \mbox{ for each } a\in A}  = \gp{H,t \mid A^t = B}$$
is called the HNN extension of $H$ relative to $\varphi$. We refer
to \cite{LS} for general facts on  HNN extensions.   The letter
$t$ is called the {\em stable letter}. If $H$ is generated by a set $Y$ then $Y \cup \{t\}$
 generates $G$ and any word $w$ in the alphabet $(Y \cup \{t\})^{\pm 1}$
 can be written in the syllable form:
    $$w(H,t) = h_0 t^{\varepsilon_1} h_1 t^{\varepsilon_2} h_2 \ldots t^{\varepsilon_n} h_n$$
where $\varepsilon_i = \pm 1$ for each $i = 1,\ldots,n$ and $h_i$
are words in the alphabet $Y^{\pm 1}$. The number $n$ is
called the {\em syllable length} of $w = w(H,t)$ and denoted by
$|w|_t$. A {\em pinch} in $w$ is a subword of the type $t^{-1}ht$ with $h \in A$ or a subword
 $tht^{-1}$ where $h \in B$. A word $w$ is {\em reduced}
if it is freely reduced and  contains no pinches.

\begin{thm}[Britton's lemma, \cite{Britton:1963}]
Let $G = \gp{H,t \mid A^t = B}$. If  a word
$$w(H,t) = h_0 t^{\varepsilon_1} h_1 t^{\varepsilon_2} h_2
\ldots t^{\varepsilon_n} h_n$$ represents the trivial element of $G$ then either
    $n = 0$ and  $w =_H 1$ or $w(H,t)$  has a pinch.

\end{thm}

\begin{cor}
\label{co:rewrite-HNN}
Let $G = \gp{H,t \mid A^t = B}$.  Assume that
\begin{itemize}
    \item[\bf(G1)]
The Word problem is solvable in $H$.
    \item[\bf(G2)]
The Membership problem is solvable for $A$ and $B$ in $H$.
    \item[\bf(G3)]
The  isomorphisms $\varphi$ and $\varphi^{-1}$ are effectively
computable.
\end{itemize}
Then the Word problem in $G$ is solvable.
\end{cor}
\begin{proof}
The decision algorithm that easily comes from the Britton's lemma
can be described as rewriting with the following infinite rewriting system $\CR_{HNN}$:
 \begin{equation}\label{eq:rewriting_system1}
\begin{array}{l}
  \{ t^{-1} h t \rightarrow \phi(h) \mid h \in F(Y) \ and \  h\in A\} ~\cup~ \\

 \{ t h t^{-1} \rightarrow \phi^{-1}(h) \mid h \in F(Y) \ and \  h \in B\}  ~\cup~ \\
 \{ h \rightarrow \varepsilon  \mid h \in F(Y) \ and \ h=_H 1\}
\end{array}
\end{equation}
 where $\varepsilon$ is the empty word.

\end{proof}

\subsection{The group $G_{(1,2)}$ and Magnus breakdown}

\begin{proposition}
\label{pro:G12}
Let $G_{(1,2)} = \gp{ a,b \mid b^{-1} a^{-1} b a b^{-1} a b = a^2 }$. Then the following holds:
\begin{itemize}
\item [1)] The group $G_{(1,2)}$ is a conjugacy pinched  HNN-extension of the
Baumslag-Solitar group $B_{(1,2)} = \gp{ a,t \mid t^{-1} a t = a^2}$ with the stable letter $b$:
    $$G_{(1,2)} =  \gp{ B_{(1,2)}, b \mid b^{-1} a b = t}$$
\item [2)] An infinite rewriting system $\CR$:
 \begin{equation}\label{eq:rewriting_system1}
\begin{array}{l}
  \{ t^{-1} a^k t \rightarrow a^{2k} \mid k \in \MZ\} ~\cup~
 \{ t a^{2k} t^{-1} \rightarrow a^{k} \mid k \in \MZ\} ~\cup~ \\
  \{ b^{-1} a^k b \rightarrow t^{k} \mid k \in \MZ\} ~\cup~
 \{ b t^{k} b^{-1} \rightarrow a^{k} \mid k \in \MZ\} ~\cup~ \\
 \{aa^{-1} \rightarrow \varepsilon,~ a^{-1}a \rightarrow \varepsilon,~ bb^{-1} \rightarrow \varepsilon,~ b^{-1}b \rightarrow \varepsilon\}
\end{array}
\end{equation}
is terminating and for any $w = w(a,b)$,
    $$w=_G1 ~~\Leftrightarrow~~ w\rightarrow_\CR^\ast \varepsilon.$$
In particular, $\CR$ gives a decision algorithm for the Word problem in $G_{(1,2)}$,
\end{itemize}
\end{proposition}
\begin{proof}

Notice that
$$G_{(1,2)} = \gp{ a,b \mid b^{-1} a^{-1} b a b^{-1} a b = a^2 } = \gp{ a,t, b \mid t^{-1} a t = a^2,~ b^{-1} a b = t} = $$
    $$= \gp{ B_{(1,2)}, b \mid b^{-1} a b = t}$$
which proves 1).

To prove 2) observe first that  the groups $B_{(1,2)}$ and $G_{(1,2)}$ are HNN extension, which
satisfy the properties (G1), (G2), and (G3) from Corollary \ref{co:rewrite-HNN}. Hence WP in both groups  can be solved by the corresponding rewriting systems of the type $\CR_{HNN}$ from the proof of Corollary \ref{co:rewrite-HNN}. Combining these rewriting systems into one we obtain the system $\CR$.  It follows that for any $w = w(a,b)$,
    $$w=_G1 ~~\Leftrightarrow~~ w\rightarrow_\CR^\ast \varepsilon.$$
It remains to be seen that $\CR$ is terminating.
To see this associate with each
word $w = w(a,b,t)$ a triple $(\alpha, \beta, \gamma)$ where
$\alpha$ is a total number of $b$ symbols in $w$, $\beta$ is the
total number of $t$ symbols in $w$, and $\gamma = |w|$. It is easy
to see that any rewrite from $\CR$ strictly decreases $(\alpha, \beta, \gamma)$
as an element of $\MN^3$ in the (left) lexicographical order.
Now termination of $\CR$ follows from the fact that $\CR$   is a
well-ordering.
\end{proof}

Notice, that the system $\CR$ is not confluent in general.

Proposition \ref{pro:G12} states that $\CR$ solves the Word
problem for $G_{(1,2)}$, but it does not give any estimate on the
time-complexity of the rewriting procedure.  To estimate the complexity of rewriting with $\CR$  consider a sequence
of words over the alphabet of $G_{(1,2)}$ defined as follows
\begin{equation}\label{eq:hard_words}
\begin{array}{ll}
w_0 = & a \\
w_1 = & (b^{-1} w_0 b)^{-1} a (b^{-1} w_0 b)\\
\ldots\\
w_{i+1} = & (b^{-1} w_i b)^{-1} a (b^{-1} w_i b)\\
\end{array}
\end{equation}

\begin{lemma}\label{le:hard_words}
Let $G = G_{(1,2)} = \gp{ a,b \mid b^{-1} a^{-1} b a b^{-1} a b = a^2 }$. Then (in the notation above) the following holds:

\begin{itemize}
\item [1)]  for any $i$
$$w_i =_G
a^{2^{\left.2^{\ldots^{2} }\right\} i~times}} = a^{tower_2(i)}
$$
\item [2)] $a^{tower_2(i)}$ is the only $\CR$-reduced form of $w_i$.

\item [3)]it takes at least $tower_2(i-1)$ elementary rewrites for $\CR$ to rewrite $w_i$ into $a^{tower_2(i)}$.

\end{itemize}
\end{lemma}

\begin{proof}
By  induction on $k$
$$w_{k+1} = (b^{-1} w_k b)^{-1} a (b^{-1} w_k b) $$
$$=_G t^{-2^{\left.2^{\ldots^{2} }\right\} k}} ~a~ t^{2^{\left.2^{\ldots^{2} }\right\} k}} =_G
a^{2^{\left.2^{\ldots^{2} }\right\} k+1}}.$$
which proves 1). Now 2) and 3) are easy.
\end{proof}

\begin{theorem}
The the time function of the Magnus breakdown algorithm on $G_{(1,2)}$ is not bounded by any finite tower of exponents.
\end{theorem}

\begin{proof}
Since the presentation $\gp{ a,b \mid b^{-1} a^{-1} b a b^{-1} a b = a^2 }$ 
for $G_{(1,2)}$ has a unique stable letter $b$ in its relator, 
the Magnus procedure represents $G_{(1,2)}$ as the HNN extension 
    $$G_{(1,2)} =  \gp{ B_{(1,2)}, b \mid b^{-1} a b = t}.$$
Similarly, the presentation $\gp{ a,t \mid t^{-1} a t = a^2}$ 
has a unique stable letter $t$ in its relator, 
the Magnus procedure represents $B_{(1,2)}$ as the HNN extension of $\MZ=\gp{a}$
    $$B_{(1,2)} =  \gp{ \gp{a}, t \mid t^{-1} a t = a^2}.$$
Now, to determine if a given word $w=w(a,b)$ represents the identity of $G_{(1,2)}$
the Magnus process applies the Britton's lemma to the constructed HNN extensions.
The rewriting system $\CR$ describes precisely the applications of the Britton's lemma
to the word $w$, when one first eliminates all the pinches 
related to $b$ and then all the pinches related to $t$. Independently of how 
one realizes the rewriting in Magnus breakdown (rewriting with $\CR$) in a 
deterministic fashion the rewriting of the words $w_i$ of (\ref{eq:hard_words}) 
is essentially unique and takes at least  $tower_2(i-1)$ elementary rewrites to finish.  
Notice that the length of the word $w_i$ is less than $2^{i+2}-1$ and, as
Lemma \ref{le:hard_words} shows, reducing the word $w_i$ produces the word of
length $tower_2(i)$.  Hence the result.
\end{proof}

\subsection{Large scale rewriting in $G_{(1,2)}$}

To make the rewriting by $\CR$  efficient one must be able
to:
\begin{itemize}
    \item
work with huge numbers that appear in powers during the
computations;
    \item
perform rewrites at bulk, i.e., perform many similar rewrites at
once.
\end{itemize}
In Section \ref{se:algorithm} we will use the rewriting system
\begin{equation}\label{eq:rewriting_system2}
\begin{array}{rl}
 \CR' &= \{b^{-1} a^m b \rightarrow t^m ,~ b t^m b^{-1} \rightarrow a^m \mid m \in \MN \} \\
 &\cup \{t^{k} a^m \rightarrow a^{m2^{-k}} t^{-k} \mid m \in \MN,~ m2^{-k} \in \MZ \} \\
 &\cup \{t^{-k} a^m \rightarrow a^{m2^k} t^{-k} \mid k \in \MN\} \\
 &\cup \{x^{k} x^{m} \rightarrow t^{k+m} \mid k,m \in \MZ,~~x\in\{a,b,t\}\} \\
\end{array}
\end{equation}
instead of the system (\ref{eq:rewriting_system1}). To perform such
rewrites efficiently one must be able to perform the following
arithmetic operations:
\begin{enumerate}
    \item[\bf(O1)]
addition and subtraction;
    \item[\bf(O2)]
multiplication and division by a power of $2$.
\end{enumerate}
In the next section we introduce a representation of integer
numbers over which the sequences of operations (O1) and (O2) can
be performed efficiently.

\section{Power circuits}
\label{se:power_circuits}

In this section we define a presentation of integers which we refer
to as {\em power circuit presentation} and show how one can perform
some arithmetic operations over power circuits. 
See \cite{Miasnikov_Ushakov_Won_1} for more details on circuits.

A power circuit is a quadruple $(\CP, \mu, M, \nu)$ satisfying the
conditions below:
\begin{enumerate}
    \item[$\bullet$]
$\CP = (V(\CP),E(\CP))$ a directed graph with no multiple edges
and no directed cycles;
    \item[$\bullet$]
$\mu:E(\CP) \rightarrow  \{ 1,-1\}$ a function called {\em the edge
labelling function};
    \item[$\bullet$]
$M \subseteq V(\CP)$ a set of vertices called {\em the set of marked
vertices};
    \item[$\bullet$]
and $\nu:M \rightarrow \{-1,1\}$ a function called {\em the sign
function}.
\end{enumerate}
For an edge $e=v_1 \rightarrow v_2$ in $\CP$
denote its origin $v_1$ by $\alpha(e)$ and its terminus $v_2$ by
$\beta(e)$. For a vertex $v$ in $\CP$ define sets
    $$In_v = \{ e\in \CP \mid \beta(e)=v \} \mbox{ and } Out_v = \{ e\in \CP \mid \alpha(e)=v \}.$$
A vertex $v$ in $\CP$ is called an {\em source} if $In_v =
\emptyset$. Inductively define a function $\CE:V(\CP) \rightarrow
\mathbb{R}$ ($\CE$ stands for evaluation) as follows: for $v\in
V(\CP)$ define
$$
\CE(v) = \left\{
\begin{array}{ll}
0 & \mbox{if } Out_v = \emptyset; \\
2^{\sum_{e\in Out_v} \mu(e)\CE(\beta(e))} & \mbox{otherwise}\\
\end{array}
\right.
$$
We are interested in presentations of integer
numbers only and hence we assume that $\CE(v) \in \MZ$ for each $v
\in \CP$. Since $\CP$ contains no cycles the function $\CE$ is
well-defined. Finally, assign a number $\CN$ to a defined
quadruple $(\CP,\mu,M,\nu)$ as follows
$$\CN = \CN(\CP,\mu,M,\nu) = \sum_{v\in M} \nu(v) \CE(v).$$
If $\CN = \CN(\CP,\mu,M,\nu)$ then we say that $(\CP,\mu,M,\nu)$ is
a {\em power circuit presentation} of the number $\CN \in
\mathbb{R}$, or that $\CN$ is represented by $(\CP,\mu,M,\nu)$.
Throughout the paper we denote the quadruple $(\CP,\mu,M,\nu)$
simply by $\CP$.

\begin{figure}[h]
\centerline{ \includegraphics[scale=0.5]{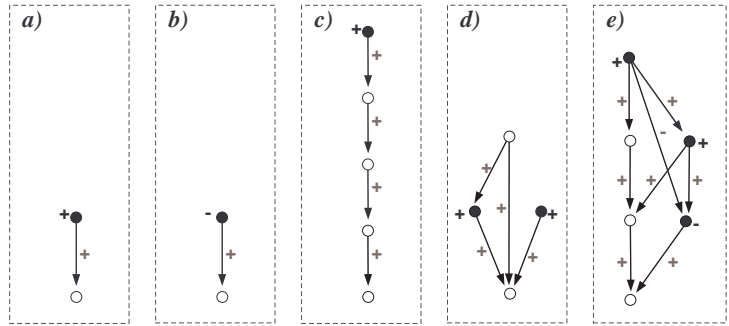} }
\caption{\label{fi:pp_examples} Power circuits
representing integers $1$, $-1$, $16$, $2$ and $35$. Black vertices
denote the marked vertices. An edge $e$ is labeled with $+$ if
$\mu(e)=1$, a marked vertex $v$ is labeled with $+$ if
$\nu(v)=1$.}
\end{figure}

For a circuit $\CP$ denote by $|\CP|$ the number $|V(\CP)| +
|E(\CP)|$ called the {\em size} of the circuit and by $\CN(\CP)$ the
integer represented by $\CP$.

\subsection{Zero vertices in power circuits}

A vertex $z$ in $\CP$ is called {\em zero} if $Out_z = \emptyset$.
It follows from the definition of the function $\CE$ that $z$ is a
zero vertex in $\CP$ if and only if $\CE(z) = 0$. Clearly, each
non-trivial circuit has at least zero vertex. If $\CP$ has more than
one zero vertex then its size can be reduced.

\begin{lemma}[\cite{Miasnikov_Ushakov_Won_1}]
Let $z_1$ and $z_2$ be distinct zero vertices of a circuit $\CP$ and
$\CP'$ a circuit obtained from $\CP$ by gluing $z_1$ and $z_2$
together. Then $|V(\CP)| = |V(\CP')|+1$ and $\CN(\CP) = \CN(\CP')$.
\end{lemma}

\subsection{Addition and subtraction}

Let $\CP_1$ and $\CP_2$ be two circuits. To compute a circuit
$\CP_+$ such that $\CN(\CP_{+}) = \CN(\CP_1) + \CN(\CP_2)$ one can
take a union of $\CP_1$ and $\CP_2$ leaving the labeling functions
the same. Clearly the obtained result satisfies the equality
$\CN(\CP_{+}) = \CN(\CP_1) + \CN(\CP_2)$. Similarly, to compute a
circuit $\CP_-$ such that $\CN(\CP_{-}) = \CN(\CP_1) - \CN(\CP_2)$
one can take a union of $\CP_1$ and $\CP_2$ leaving the labeling
functions on $\CP_1$ the same and changing the labeling function on
$M(\CP_2)$ to the opposite. Clearly the obtained result satisfies
the required equality. See Figure \ref{fi:difference} for an example
of difference of two circuits.

\begin{figure}[h]
\centerline{ \includegraphics[scale=0.5]{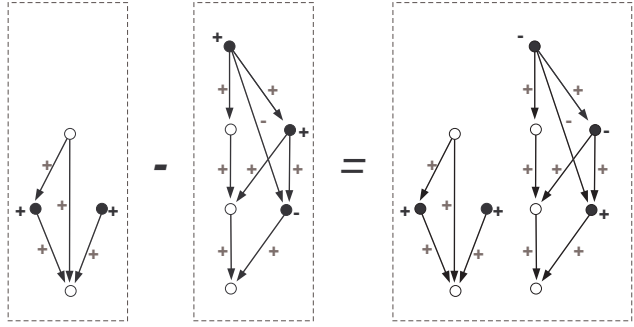} }
\caption{\label{fi:difference} Taking difference of circuits d) and e) from
Figure \ref{fi:pp_examples}.}
\end{figure}

\begin{proposition}[\cite{Miasnikov_Ushakov_Won_1}] \label{pr:complexity_add}
Let $\CP_1$ and $\CP_2$ be power circuits and $\CP_+ = \CP_1 +
\CP_2$. Then $\CN(\CP_+) = \CN(\CP_1) + \CN(\CP_2)$, $|V(\CP_+)| =
|V(\CP_1)| + |V(\CP_2)|$, and $|E(\CP_+)| = |E(\CP_1)| +
|E(\CP_2)|$. Moreover, $\CP_+$ and $\CP_-$ are computed in time
$O(|\CP_1| + |\CP_2|)$.
\end{proposition}

\subsection{Comparison (circuit reduction)}
\label{se:comparison}

In this section we shortly describe the procedure called the {\em
reduction} of power circuits. For the precise definition of a
reduced circuit see \cite{Miasnikov_Ushakov_Won_1}. The main property of reduced
circuits is that $\CE(v_1)=\CE(v_2)$ if and only if $v_1 = v_2$.

\begin{theorem}[\cite{Miasnikov_Ushakov_Won_1}]\label{th:comparison}
There exists an algorithm which for every power circuit $\CP$
constructs an equivalent reduced circuit $\CP'$ such that
$$|V(\CP')| \le |V(\CP)|+1 \mbox{ and } |M(\CP')| \le |M(\CP)|,$$
and orders vertices of $\CP'$ according to their $\CE$ values.
Moreover, the time complexity of the procedure is $O(|V(\CP)|^3)$.
\end{theorem}

\begin{proposition}[\cite{Miasnikov_Ushakov_Won_1}]
Let $\CP$ be a reduced circuit. Then $\CN(\CP) = 0$ if and only if
$\CP$ has no marked vertices. If $\CP$ is not trivial and if $v$
is the vertex with maximal $\CE$ value then $\CN(\CP) > 0$ if and
only if $\nu(v) = 1$.
\end{proposition}

\begin{proposition}[\cite{Miasnikov_Ushakov_Won_1}] \label{pr:complexity_compare_circ}
There exists a deterministic algorithm which for every power circuit
$\CP$ computes
$$Sign(\CP) =
 \left\{
 \begin{array}{ll}
 -1, & \mbox{if } \CN(\CP)<0;\\
 0, & \mbox{if } \CN(\CP)=0;\\
 1, & \mbox{if } \CN(\CP)>0.\\
 \end{array}
 \right.
$$
Moreover, the time complexity of that procedure is bounded above by
$O(|V(\CP)|^3)$.
\end{proposition}

\begin{proposition}[\cite{Miasnikov_Ushakov_Won_1}]\label{pr:divisibility}
Let $\CP$ be a reduced power circuit, $\{v_1,\ldots,v_n\}$ a set
of all its vertices ordered according to their $\CE$ values, and
$i = \min\{j \mid v_j \in M\}$. Then $\CN(\CP)$ is divisible by
$2^M$ if and only if $\CE(v_i)$ is.
\end{proposition}

\subsection{Multiplication and division by a power of two}
\label{se:multiplication_power}

Let $\CP_1$ and $\CP_2$ be power circuits. Assume that
$\CN(\CP_2)>0$. In this section we outline a procedure for
constructing circuits $\CP_\bullet$ and $\CP_\circ$ satisfying
    $$\CN(\CP_\bullet) = \CN(\CP_1) \bullet \CN(\CP_2) := \CN(\CP_1) \cdot 2^{\CN(\CP_2)}$$
and
    $$\CN(\CP_\circ) = \CN(\CP_1) \circ \CN(\CP_2) := \frac{\CN(\CP_1)}{2^{\CN(\CP_2)}}.$$
Recall that $\CN(\CP_1) = \sum_{v\in M_1} \nu(v) \CE(v)$, where
$\CE(v) = 2^{\sum_{e\in Out_v} \mu(e)\CE(\beta(e))}$ is a power of
$2$. Hence, to multiply $\CN(\CP_1)$ by $2^{\CN(\CP_2)}$ one can
multiply the values of $\CE(v)$ by $2^{\CN(\CP_2)}$ for each $v \in
M(\CP_1)$ which corresponds to increase of the value of the sum
$\sum_{e\in Out_v} \mu(e)\CE(\beta(e))$ by $\CN(\CP_2)$. Thus, to
multiply $\CN(\CP_1)$ by $2^{\CN(\CP_2)}$ one can perform the
following steps:
\begin{enumerate}
    \item[(1)]
make each marked vertex $v$ in $\CP_1$ a source;
    \item[(2)]
take a union of $\CP_1$ and $\CP_2$;
    \item[(3)]
for each $v_1 \in M_1$ and $v_2 \in M_2$ add an edge $e = v_1
\rightarrow v_2$ and put $\mu(e) = \nu(v_2)$;
    \item[(4)]
unmark marked vertices of $\CP_2$.
\end{enumerate}
See Figure \ref{fi:mult_power2} for an example.

In this paper we work with integer numbers only. Hence the
operation $\circ$ is not always defined for all pairs $\CP_1$,
$\CP_2$ of circuits. To check if $\CP_1 \circ \CP_2$ is defined
one can reduce the presentation of $\CP_1$ and check the
conditions of Proposition \ref{pr:divisibility}. To actually
multiply $\CP_1$ by $2^{-\CN(\CP_2)}$ one needs to 1) reduce
$\CP_1$, 2) invert the value of $\CP_2$ and, 3) apply the
algorithm outlined above to compute $\CP_\circ$.

\begin{figure}[htbp] \centerline{
\includegraphics[scale=0.5]{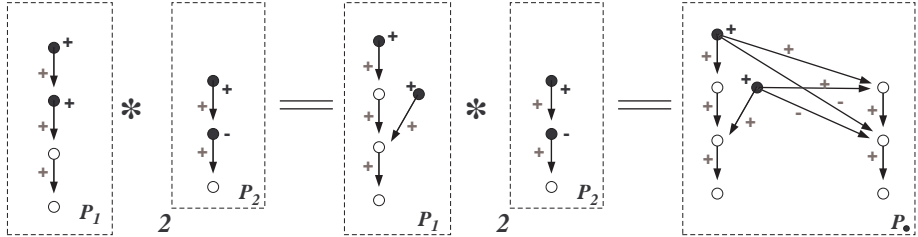} }
\caption{\label{fi:mult_power2} Multiplication by a power of $2$.}
\end{figure}

\begin{proposition}\label{pr:complexity_divpower2}
Let $\CP_1$ and $\CP_2$ be power circuits and, $\CP_\bullet$ and
$\CP_\circ$ are obtained by the outlined above procedures. Then
\begin{itemize}
    \item[1)]
$\CN(\CP_\bullet) = \CN(\CP_1) 2^{\CN(\CP_2)}$ and $\CN(\CP_\circ) =
\frac{\CN(\CP_1)} {2^{\CN(\CP_2)}}$;
    \item[2)]
$|V(\CP_\bullet)|,|V(\CP_\circ)| \le |V(\CP_1)|+|V(\CP_2)| +
|M_1|$.
    \item[3)]
The time required to construct  $\CP_\bullet$ is bounded by
$O(|\CP_1|+|\CP_2|)$.
    \item[4)]
The time required to construct  $\CP_\circ$ is bounded by
$O(|V(\CP_1)|^3+|\CP_2| + |M_1|\cdot|M_2|)$.
\end{itemize}
\end{proposition}

\begin{proof}
Straightforward to check.
\end{proof}

\section{Power sequences}
\label{se:power_sequences}

Let $X$ be a group alphabet, $x_1,\ldots,x_n \in X^{\pm 1}$,
$\CP_1,\ldots,\CP_n$ power circuits.
A sequence $\CS = (x_1,\CP_1),\ldots, (x_n,\CP_n)$
is called a {\em power sequence}.
We say that a power sequence $\CS$ represents a word
    $$W(\CS) = x_1^{\CN(\CP_1)} \ldots x_n^{\CN(\CP_n)}.$$
The following characteristics of power sequences are used in our
analysis in Section \ref{se:algorithm}. We denote by $M(\CS)$
the total number of marked vertices in its circuits, i.e.,
    $$M(\CS) = \sum_{(x,\CP) \in \CS} |M(\CP)|,$$
and $V(\CS)$ the total number of vertices in its circuits, i.e.,
    $$V(\CS) = \sum_{(x,\CP) \in \CS} |V(\CP)|.$$
If $\CS$ represents a word $w = x_1^{p_1} \ldots x_n^{p_n}$
and $g = x_i^{p_i} \ldots x_j^{p_j}$ is a subword of $w$ denote by
$\CS_g$ the segment of $\CS$ corresponding to $g$.

A power sequence is {\em reduced} if it does not contain
\begin{enumerate}
    \item[\bf(R1)]
a pair $(x,\CP)$ where $\CN(\CP)=0$,
    \item[\bf(R2)]
a subsequence $(x,\CP),(x,\CP')$.
\end{enumerate}
To reduce a power sequence $\CS$ one can consequently replace
non-reduced subsequences $(x,\CP),(x,\CP')$ by the corresponding
pairs $(x,\CP + \CP')$, and remove the
pairs (R1). The described process is called a {\em reduction} of a power
sequence.

\begin{proposition}\label{pr:reduction}
Let $\CS$ be a power sequence and $\CS'$ be obtained by reducing
$\CS$. Then $\CS$ and $\CS'$ represents the same element of the
corresponding free group $F(A)$. Furthermore, $M(\CS') \le M(\CS)$
and $V(\CS') \le V(\CS)$. The time complexity of reduction is not
greater than $O(|V(\CS)|^3)$.
\end{proposition}

\begin{proof}
Follows from Proposition \ref{pr:complexity_add}.

\end{proof}

\section{Algorithm for the Word problem in $G_{(1,2)}$}
\label{se:algorithm}

In this section we describe an algorithm for the Word problem in $G_{(1,2)}$ and
prove that it has polynomial time complexity. All words in this
section are processed in power sequences.
More details are given in Section \ref{se:BS_WordProblem}.
In Section \ref{se:complexity} we give the final complexity estimate.

\begin{algorithm}{Word problem for $G_{(1,2)}$}\label{al:WordProblem}
\begin{algorithmic}[1]
\REQUIRE A word $w = w(a,b,t)$.
\ENSURE $Yes$ if $w$ represents the identity in $G_{(1,2)}$, $No$ otherwise.
\STATE Represent $w$ as a product of powers
\begin{equation}\label{eq:w}
    w(a,b,t) = g_0(a,t) b^{\varepsilon_1} g_1(a,t) b^{\varepsilon_2} g_2(a,t) \ldots b^{\varepsilon_n} g_n(a,t)
\end{equation}
where
\begin{equation}\label{eq:g_i}
    g_i(a,t) = a^{m_{i,0}} t^{\delta_{i,1}} a^{m_{i,1}} t^{\delta_{i,2}} a^{m_{i,2}} \ldots t^{\delta_{i,k_i}} a^{m_{i,k_i}}
\end{equation}
and $\varepsilon_i,\delta_{i,j},m_{i,j} \in \MZ$.
\STATE Compute a power sequence $\CS$ representing $w$.
\WHILE{$\CS$ contains a subsequence $\CS_{g_i}$ satisfying the follwoing} \label{step:while}
    \IF{$\varepsilon_i<0$, $\varepsilon_{i+1}>0$, and $\CS_{g_i} =_{B} a^p$ for some $p\in \MZ$}
        \STATE \label{step:red1} Replace a subsequence $(b,\varepsilon_i) \CS_{g_i} (b,\varepsilon_{i+1})$ in $\CS$ with $(b,\varepsilon_i+1),(t,p),(b,\varepsilon_{i+1}-1)$.
    \ENDIF
    \IF{$\varepsilon_i>0$, $\varepsilon_{i+1}<0$, and $\CS_{g_i} =_{B} t^p$ for some $p\in \MZ$}
        \STATE \label{step:red2} Replace a subsequence $(b,\varepsilon_i) \CS_{g_i} (b,\varepsilon_{i+1})$ in $\CS$ with $(b,\varepsilon_i+1),(a,p),(b,\varepsilon_{i+1}-1)$.
    \ENDIF
\ENDWHILE\label{step:end_while}
\IF{The obtained $\CS$ involves letter $b$}
    \RETURN $No$.
\ENDIF
\IF{If the obtained sequence represents the trivial element in $B_{(1,2)}$}
    \RETURN $Yes$.
\ELSE
    \RETURN $No$.
\ENDIF
\end{algorithmic}
\end{algorithm}

A single transformation on line \ref{step:red1} and line \ref{step:red2}
decreases the total power of $b$
in a power sequence $\CS$ by $2$. Hence Algorithm \ref{al:WordProblem}
performs at most $|w|/2$ transformations. Now, it remains to describe a procedure
for checking if $\CS_{g_i} =_G a^p$ or $\CS_{g_i}=_G t^p$ for some $p\in \MZ$ for some $p\in \MZ$.
This is done in Section \ref{se:BS_WordProblem}.

\subsection{Word processing in $B_{(1,2)}$}
\label{se:BS_WordProblem}

Let
\begin{equation}\label{eq:BS_word}
    (a,\CP_{m_{0}}), (t,\CP_{\delta_{1}}), (a,\CP_{m_{1}}), (t,\CP_{\delta_{2}}), (a,\CP_{m_{2}}) \ldots (t,\CP_{\delta_{k}}), (a,\CP_{m_{k}})
\end{equation}
be a power sequence representing a word
    $$g = a^{m_{0}} t^{\delta_{1}} a^{m_{1}} t^{\delta_{2}} a^{m_{2}} \ldots t^{\delta_{k}} a^{m_{k}}$$
over the alphabet of $B_{(1,2)}$.

\begin{proposition}[All non-positive powers]\label{pr:negative_t}
Consider a sequence (\ref{eq:BS_word}).
Assume that $\delta_1+\ldots+\delta_i \le 0$ for every $i=1,\ldots,k$.
Then $g = a^M t^\sigma$ in $G_{(1,2)}$ where
    $$\sigma = \sum_{i=1}^k \delta_{i} \mbox{ and } M = \sum_{i=0}^k m_i \cdot 2^{\sum_{j=1}^{i}\delta_j}.$$
Furthermore, there exist power circuits $\CP_M$ and $\CP_\sigma$ such that
$\CN(\CP_M) = M$ and $\CN(\CP_\sigma) = \sigma$ and
\begin{enumerate}
    \item[(1)]
$|V(\CP_{\sigma})| = \sum_{j=1}^{k} |V(\CP_{\delta_j})|$ and $|V(\CP_{M})| \le \sum_{j=0}^{k} (|V(\CP_{m_j})|+|M(\CP_{m_j})|) + \sum_{j=1}^{k} |V(\CP_{\delta_j})|$;
    \item[(2)]
$|M(\CP_{\sigma})| = \sum_{j=1}^{k} |M(\CP_{\delta_j})|$
and
$|M(\CP_{M})| = \sum_{j=0}^{k} |M(\CP_{m_j})|.$
\end{enumerate}
\end{proposition}

\begin{proof}
The equality $g = a^M t^\sigma$ in $G_{(1,2)}$ is obvious.
A circuit $\CP_\sigma$ can be constructed by laying down
circuits for $\CP_{\delta_{1}},\ldots,\CP_{\delta_{k}}$.
A circuit $\CP_M$ is obtained by
\begin{itemize}
    \item
laying down the circuits $\CP_{m_0},\ldots,\CP_{m_k}$,
$\CP_{\delta_1},\ldots,\CP_{\delta_k}$;
    \item
adding edges between $\CP_{m_i}$ and $\CP_{\delta_j}$ as it is done in
multiplication by a power of two (see Section \ref{se:multiplication_power});
    \item
removing marked vertices from $\CP_{\delta_j}$'s.
\end{itemize}
See Figure \ref{fi:long_sum}.
\begin{figure}[htbp] \centerline{
\includegraphics[scale=0.5]{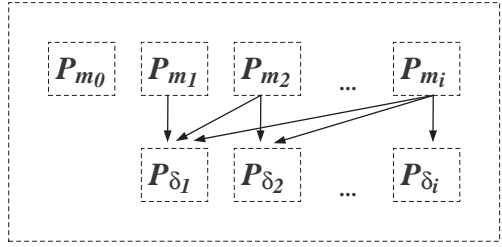} }
\caption{\label{fi:long_sum} Scheme for a circuit
representing $\sum_{i=0}^k \CP_{m_{i}} \circ \rb{\sum_{j=1}^i \CP_{\delta_{j}}}$.}
\end{figure}
Clearly $\CP_M$ and $\CP_\sigma$ satisfy properties (1) and (2).
\end{proof}

We call the transformation of Proposition \ref{pr:negative_t} a (T1)-transformation.
Applying (T1)-transformations to subsequences of (\ref{eq:BS_word}) we either obtain
a power sequence
\begin{equation}\label{eq:after1}
    (a,\CP_{M_{0}}), (t,\CP_{\sigma_{1}}), (a,\CP_{M_{1}}), \ldots ,(a,\CP_{M_{n-1}}),(t,\CP_{\sigma_{n}}), (a,\CP_{M_{n}}),\phantom{ ,(t,\CP_{\sigma_{n+1}})}
\end{equation}
where $\sigma_{i}>0$ for every $i=1,\ldots,n$;
or a power sequence
\begin{equation}\label{eq:after2}
    (a,\CP_{M_{0}}), (t,\CP_{\sigma_{1}}), (a,\CP_{M_{1}}), \ldots ,(a,\CP_{M_{n-1}}),(t,\CP_{\sigma_{n}}), (a,\CP_{M_{n}}) ,(t,\CP_{\sigma_{n+1}}),
\end{equation}
where $\sigma_{i}>0$ for every $i=1,\ldots,n$ and $\sigma_{n+1}<0$. Furthermore, for
(\ref{eq:after1}) there exist a sequence $1\le l_0<l_1<\ldots<l_n = k$
such that
\begin{equation}\label{eq:powers_after}
    \CP_{\sigma_i} = \CP_{\delta_{l_i}}+\ldots+\CP_{\delta_{l_{i+1}-1}}\mbox{ and } \CP_{M_i} =
    \sum_{i=l_i-1}^{l_{i+1}-1} \CP_{m_{i}} \circ \rb{\sum_{j=l_i}^{i} \CP_{\delta_{j}}}.
\end{equation}
Similar formulas hold for (\ref{eq:after2}).
The next lemma follows from the Britton's lemma.

\begin{lemma}
If a sequence (\ref{eq:after1}) or (\ref{eq:after2})
represents in $B_{(1,2)}$ an element $a^p$ or $t^p$ for some $p\in\MZ$
then for every $i=1,\ldots,n$ the condition
    $$2^{-\sigma_{i}}\rb{\ldots + 2^{-\sigma_{n-2}}\rb{M_{n-2}+2^{-\sigma_{n-1}}\rb{M_{n-1}+2^{-\sigma_n}M_n}}} \in \MZ$$
is satisfied.
\end{lemma}

\begin{proposition}[All positive powers]\label{pr:positive_t}
Consider a sequence (\ref{eq:after1}) or (\ref{eq:after2})
representing an element $g$ in $B_{(1,2)}$.
If for every $i=1,\ldots,n$ the condition
    $$2^{-\sigma_{i}}\rb{\ldots + 2^{-\sigma_{n-2}}\rb{M_{n-2}+2^{-\sigma_{n-1}}\rb{M_{n-1}+2^{-\sigma_n}M_n}}} \in \MZ$$
is satisfied then $g = a^M t^\sigma$, where $\sigma = \sum_{i=0}^n \sigma_i$ and
    $$M = \rb{M_{0}+\ldots \rb{M_{n-2}+\rb{M_{n-1}+M_n\circ \sigma_n}\circ \sigma_{n-1}} \ldots \circ \sigma_{1}}.$$
Furthermore, there exist power circuits $\CP_M$ and $\CP_\sigma$ such that
$\CN(\CP_M) = M$ and $\CN(\CP_\sigma) = \sigma$ satisfying
\begin{enumerate}
    \item[(1)]
$|V(\CP_{\sigma})| = \sum_{j=1}^{n} |V(\CP_{\sigma_j})|$ and $|V(\CP_{M})| \le \sum_{j=0}^{n} (|V(\CP_{M_j})|+|M(\CP_{M_j})|) + \sum_{j=1}^{n} |V(\CP_{\sigma_j})|$;
    \item[(2)]
$|M(\CP_{\sigma})| = \sum_{j=1}^{n} |M(\CP_{\sigma_j})|$
and
$|M(\CP_{M})| = \sum_{j=0}^{k} |M(\CP_{M_j})|.$
\end{enumerate}
\end{proposition}

\begin{proof}
The equality $g = a^M t^\sigma$ follows from the Britton's lemma.
To construct a circuit $\CP_\sigma$ we lay down
circuits for $\CP_{\sigma_1},\ldots,\CP_{\sigma_n}$.
Clearly, $\CP_\sigma$ satisfies the properties (1) and (2).

We construct a circuit $\CP_M$ by induction on $n$. If $n=0$ then $\CP_M = \CP_{M_0}$
and we have nothing to do.
The case when $n=1$ provides us with the induction step. If $n=1$ then we
need to construct a circuit representing $M_{0}+2^{-\sigma_1}M_1$.
By (\ref{eq:powers_after}) we have $\sigma_1 = \delta_0+\ldots+\delta_l$
and $M_0 = m_0 + m_1 2^{-\delta_1}+\ldots+ m_{l-1}s^{-\delta_1\ldots-\delta_{l-1}}$.
A structure of circuits $\CP_{\sigma_1}$ and $\CP_{M_0}$ was described in Proposition
\ref{pr:negative_t}. To construct $\CP_M$ we
\begin{itemize}
    \item
put down $\CP_{M_0},\CP_{\delta_{l}}$ and a reduced $\CP_{M_1}$;
    \item
make all vertices in $\CP_{\delta_{l}}$ unmarked;
    \item
($\CP_{M_0}$ contains subgraphs corresponding to $\CP_{\delta_0},\ldots,\CP_{\delta_{l-1}}$)
add edges from marked vertices in $\CP_{M_1}$ to vertices in
$\CP_{\delta_0},\ldots,\CP_{\delta_{l}}$ that were marked as for operation $\circ$;
    \item
collapse zero-vertices in the obtained graph.
\end{itemize}
It follows from the construction that $\CN(\CP_M) = M$
and that properties (1) and (2) hold for the constructed $\CP_M$.
The scheme for $\CP_M$ is given in Figure \ref{fi:long_sum_add}.
\begin{figure}[htbp] \centerline{
\includegraphics[scale=0.5]{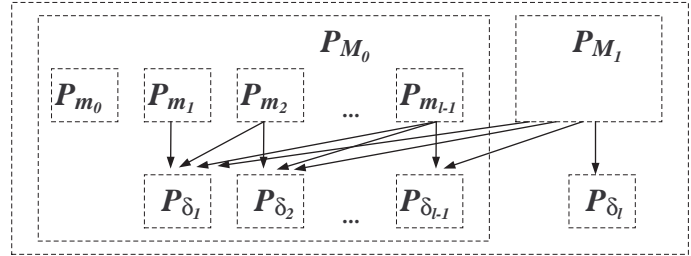} }
\caption{\label{fi:long_sum_add} Scheme for a circuit representing $M_{0}+2^{-\sigma_1}M_1$.}
\end{figure}
\end{proof}

\begin{proposition}[Complexity of processing in $B_{(1,2)}$]\label{pr:bs_final_words}
It takes
    $$O\rb{k \rb{\sum_{j=0}^{k} (|V(\CP_{m_j})|+|M(\CP_{m_j})|) + \sum_{j=1}^{k} |V(\CP_{\delta_j})|}^3}$$
operations to determine if (\ref{eq:BS_word}) is equivalent to a sequence $(a,\CP)$ or a sequence
$(t,\CP)$. If (\ref{eq:BS_word}) is equivalent to a sequence $(a,\CP)$
then $\CP$ satisfies:
    $$|V(\CP)| \le \sum_{j=0}^{k} (|V(\CP_{m_j})|+|M(\CP_{m_j})|) + \sum_{j=1}^{k} |V(\CP_{\delta_j})| \mbox{ and } |M(\CP)| = \sum_{j=0}^{k} |M(\CP_{m_j})|.$$
If (\ref{eq:BS_word}) is equivalent to a sequence $(t,\CP)$
then $\CP$ satisfies:
    $$|V(\CP)| = \sum_{j=1}^{k} |V(\CP_{\delta_j})| \mbox{ and } |M(\CP)| = \sum_{j=1}^{k} |M(\CP_{\delta_j})|.$$
\end{proposition}

\begin{proof}
The bounds on $|V(\CP)|$ and $|M(\CP)|$ for both cases follow from Propositions
\ref{pr:negative_t} and \ref{pr:positive_t}. Furthermore, at every step in the process
all power circuits in the sequence (\ref{eq:BS_word}) have the number of vertices
bounded by $\sum_{j=0}^{k} (|V(\CP_{m_j})|+|M(\CP_{m_j})|) + \sum_{j=1}^{k} |V(\CP_{\delta_j})|$.
Hence, it takes up to
    $$O\rb{\rb{\sum_{j=0}^{k} (|V(\CP_{m_j})|+|M(\CP_{m_j})|) + \sum_{j=1}^{k} |V(\CP_{\delta_j})|}^3}$$
operations to check if conditions of Propositions \ref{pr:negative_t} and \ref{pr:positive_t} hold at every step.
The algorithm performs $O(k)$ transformations and hence the claimed bound on complexity.
\end{proof}

\subsection{Complexity estimate for Algorithm \ref{al:WordProblem}}
\label{se:complexity}

\begin{theorem}
Algorithm \ref{al:WordProblem} solves the Word problem for $G_{(1,2)}$
in time $O(|w|^7)$.
\end{theorem}

\begin{proof}
Let $w = w(a,b,t)$ be a reduced word over the alphabet $\{a,b,t\}$.
First, Algorithm \ref{al:WordProblem} constructs a power sequence
$\CS$ for $w$. As described in \cite{Miasnikov_Ushakov_Won_1}
it is straightforward to construct circuits for numbers
$m_i,\varepsilon_i,\delta_i$. Clearly, the total number of
vertices for circuits $m_i,\varepsilon_i,\delta_i$ is not greater
than $2|w|$. This can be done in $O(|w|)$ steps.

In the loop \ref{step:while}--\ref{step:end_while}
Algorithm \ref{al:WordProblem} determines
what subsequences $\CS_{g_i}$ can be shortened into $(a,\CP)$ or $(t,\CP)$.
By Proposition \ref{pr:bs_final_words} this can be done in time
    $$O\rb{|g_i|\cdot \rb{|V(\CP_{g_i})|+|M(\CP_{g_i})|}^3}$$ 
and the obtained circuit $\CP$ satisfies
    $$|V(\CP)| \le V(\CS_{g_i})+M(\CS_{g_i}) \mbox{ and } |M(\CP)| \le M(\CS_{g_i}).$$
Hence, a single transformation on a step \ref{step:red1} or \ref{step:red2}
\begin{itemize}
    \item
does not increase the total number of marked vertices in $\CS$;
    \item
can increase the total number of vertices by the number of marked vertices.
\end{itemize}
Therefore, in the worst case steps \ref{step:red1} and \ref{step:red2}
are performed on a sequence $S_{g_i}$ of size $|V(S_{g_i})| = O(|w|^2)$. Algorithm
\ref{al:WordProblem} performs up to $|w|$ steps \ref{step:red1} and \ref{step:red2}.
Hence the result.
\end{proof}

\providecommand{\bysame}{\leavevmode\hbox to3em{\hrulefill}\thinspace}
\providecommand{\MR}{\relax\ifhmode\unskip\space\fi MR }
\providecommand{\MRhref}[2]{%
  \href{http://www.ams.org/mathscinet-getitem?mr=#1}{#2}
}
\providecommand{\href}[2]{#2}

\end{document}